\documentclass[a4paper,11pt]{amsart}
\usepackage{amsfonts,amssymb,amsthm,cite,amsmath,amstext}
\usepackage{color}
\usepackage[colorlinks,linkcolor=black,citecolor=black]{hyperref}
\usepackage{comment}
\usepackage{framed}

\definecolor{shadecolor}{gray}{0.875}

\newtheorem{thrm}{Theorem}[section]
\newtheorem{lem}[thrm]{Lemma}
\newtheorem{cor}[thrm]{Corollary}
\newtheorem{prop}[thrm]{Proposition}

\theoremstyle{definition}

\newtheorem{exmple}[thrm]{Example}
\newtheorem{rmk}[thrm]{Remark}

\DeclareMathOperator{\vol}{vol}

\title{B\'{e}zout type inequality in convex geometry}
\author{Jian Xiao}
\date{}

\begin{document}

\begin{abstract}
We give a B{\'e}zout type inequality for mixed volumes, which holds true for any convex bodies. The key ingredient is the reverse Khovanskii-Teissier inequality for convex bodies, which was obtained in our previous work and inspired by its correspondence in complex geometry.
\end{abstract}

\maketitle

\tableofcontents

\section{Introduction}
\subsection{Motivation}
The classical B\'ezout's theorem in algebraic geometry is concerning the number of intersection points of algebraic hypersurfaces, which do not have infinitely many common points. B\'ezout's theorem gives an upper bound for the number of intersection points of such hypersurfaces.
More precisely, let $H_1, ..., H_n \subset \mathbb{C}^n$ be algebraic hypersurfaces of respective degrees $d_1, ..., d_n$, and assume that the hypersurfaces have isolated intersection points, then the intersection number is at most $d_1\cdots d_n$. B\'{e}zout's theorem has various generalizations. For example, the Bernstein-Kushnirenko-Khovanskii theorem (see \cite{bernsteinROOT}, \cite{khovanBKK}, \cite{kusniBKK}) estimates the number of intersection points of $n$ hypersurfaces in $(\mathbb{C}^*)^n$ with fixed Newton polytopes and generic coefficients. Inspired by these two theorems, the following new geometric inequality for convex bodies was obtained in \cite{sz16bezout}:
\begin{equation}\label{eq sz}
  V(K_1,...,K_r, \Delta^{n-r})\vol(\Delta)^{r-1} \leq \prod_{i=1}^r V(K_i, \Delta^{n-1}),
\end{equation}
where $\Delta$ is an $n$-dimensional simplex and $K_1,...,K_r$ are arbitrary convex bodies. Let us present its quite simple and elegant proof. Let $H_1,...,H_r \subset (\mathbb{C}^*)^n$ be generic hypersurfaces with respective Newton polytopes $P_1,...,P_r$. Let $H_{r+1},...,H_{n}$ be given by generic linear forms, thus their Newton polytopes are the standard $n$-simplex $\Delta$. By Bernstein-Kushnirenko-Khovanskii theorem, we have
\begin{equation*}
  \#(H_1 \cap ...\cap H_n)=n!V(P_1,...,P_r, \Delta^{n-r}).
\end{equation*}
On the other hand, $\deg(H_i) = n!V(P_i, \Delta^{n-1})$ for $1\leq i \leq r$, and $\deg(H_i) = 1$ for $i\geq r+1$. Then B\'{e}zout's bound implies
\begin{equation*}
  n!V(P_1,...,P_r, \Delta^{n-r}) =  \#(H_1 \cap ...\cap H_n) \leq \prod_{i=1}^n \deg(H_i)= \prod_{i=1}^r n!V(P_i, \Delta^{n-1}).
\end{equation*}
Since $\vol(\Delta)=1/n!$, B\'{e}zout's theorem implies the above inequality (at least for Newton polytopes).

The authors called the above inequality \textsf{B\'{e}zout inequality for mixed volumes}. In particular, this inequality implies the classical B\'{e}zout's bound for hypersurfaces in $(\mathbb{C}^*)^n$. Moreover, it is conjectured that if the above inequality is true for all convex bodies $K_1,...,K_r$, then $\Delta$ must be an $n$-dimensional simplex. This conjecture was proved in \cite{sszsimplices} if $\Delta$ varies in the class of convex polytopes. Thus the B\'{e}zout inequality characterizes simplices in the class of convex polytopes.
Besides the above inequality, \cite{sz16bezout} also discussed an isomorphic version of the inequality: there is a constant $c_{n, r}$ (depending only on $n$ and $r$) such that
\begin{equation*}
  V(K_1, K_2,...,K_r, D^{n-r})\vol(D)^{r-1} \leq c_{n,r}\prod_{i=1} ^r V(K_i, D^{n-1})
\end{equation*}
is true for all convex bodies $K_1, K_2,...,K_r, D$ in $\mathbb{R}^n$. Furthermore, when $K_1, K_2,...,K_r$ are zonoids, then $c_{n, r}$ can take to be ${r^r}/{r!}$ (see \cite{sz17bezoutcorr}) and this constant is sharp in this case.

\subsection{Main Result}
The aim of this note is to give the following B\'{e}zout type inequality for mixed volumes, which holds true for arbitrary convex bodies.

\begin{thrm}\label{main thrm}
Assume that $a_1,...,a_r \in \mathbb{N}$ and $|a|:=\sum_{i=1} ^r a_i \leq n$, then for any convex bodies $K_1, ..., K_r, D \subset \mathbb{R}^n$ and any integer $k$ satisfying $1\leq k\leq r$, we have
\begin{equation*}
 \left(\begin{array}{c} n\\ a_k \end{array} \right) V(K_1 ^{a_1}, K_2 ^{a_2},...,K_r ^{a_r}, D^{n-|a|})\vol(D)^{r-1} \leq \prod_{i=1} ^r \left(\begin{array}{c} n\\ a_i \end{array} \right) V(K_i ^{a_i}, D^{n-a_i}).
\end{equation*}
\end{thrm}

As a special case, we get the following B\'{e}zout inequality for arbitrary convex bodies, which improves the corresponding estimate in \cite{sz16bezout}, where the authors gave an upper bound: $c_{n,r} \leq n^r r^{r-1}$.

\begin{cor}\label{main cor}
For any convex bodies $K_1, ..., K_r, D \subset \mathbb{R}^n$, we have
\begin{equation*}
   V(K_1, K_2,...,K_r, D^{n-r})\vol(D)^{r-1} \leq n^{r-1}\prod_{i=1} ^r V(K_i, D^{n-1}).
\end{equation*}

\end{cor}

\begin{rmk}
For convex bodies in $\mathbb{R}^2$, the inequality in Corollary \ref{main cor} was also noticed in \cite[Section 5]{AFO2dim} and \cite[Section 6]{sz16bezout}.
\end{rmk}

\begin{exmple}\label{rmk eq conv}
For convex bodies in $\mathbb{R}^2$, the inequality is sharp: let $K, L$ be line segments and let $D=K+L$, then $V(K, L)\vol(D)=2 V(K, D)V(L, D)$ by the additivity of the mixed volumes.
\end{exmple}

\begin{rmk}
By the Bernstein-Kushnirenko-Khovanskii theorem, Theorem \ref{main thrm} implies an inequality for the number of solutions of generic Laurent polynomials with fixed Newton polytopes. More precisely, let $\{P_{i, j}\}_{j=1} ^{a_i}, \{Q_m\}_{m=1}^{n}$ be generic Laurent polynomials with respective Newton polytopes $K_i$ and $D$. Then
\begin{equation*}
  \left(\begin{array}{c} n\\ a_k \end{array} \right) N(K_1 ^{a_1}, K_2 ^{a_2},...,K_r ^{a_r}, D^{n-|a|})N(D)^{r-1} \leq \prod_{i=1} ^r \left(\begin{array}{c} n\\ a_i \end{array} \right) N(K_i ^{a_i}, D^{n-a_i}),
\end{equation*}
where $N(\cdot)$ denotes the number of solutions of generic Laurent polynomials with corresponding Newton polytopes.
\end{rmk}

\subsubsection{Application}
From the correspondences between complex geometry and convex geometry, we also have the analogy of Theorem \ref{main thrm} on projective varieties.
Then we get a generalization of the classical B\'{e}zout theorem for algebraic hypersurfaces to complete intersection subvarieties.

\begin{thrm}
Let $X$ be an $n$-dimensional complex projective manifold with polarization $H$ and let $A_1,..., A_r$ be nef divisor classes. Assume that $a_1,...,a_r \in \mathbb{N}$ and $|a|:=\sum_{i=1} ^r a_i \leq n$. Let $Y_1,...,Y_r$ be subvarieties of cycle classes $A_1 ^{a_1},...,A_r ^{a_r}$, and assume they have proper intersection, then
\begin{equation*}
  \deg(Y_1 \cap...\cap Y_r) \leq \min_k \left\{\frac{\prod_{i=1} ^r \left(\begin{array}{c} n\\ a_i \end{array} \right)}{\left(\begin{array}{c} n\\ a_k \end{array} \right)(H^n)^{r-1}}\right\} \prod_{i=1}^r \deg(Y_i).
\end{equation*}
\end{thrm}

\begin{rmk}
For the constant in the B\'ezout type inequality, as shown by (\ref{eq sz}), the best possible one depends on the conditions that $K_1, ...,K_r, D$ (or $A_1, ...,A_r, H$) satisfy. For instance, if $K_1 = K_2 =...=K_r=K$ and $a_1 =...=a_r =1$, then we need to compare $V(K^{r}, D^{n-r}) \vol(D)^{r-1}$ and $V(K, D^{n-1})^r$. By Alexandrov-Fenchel inequality, we actually have
\begin{equation*}
  V(K^{r}, D^{n-r})\vol(D)^{r-1} \leq V(K,D^{n-1})^r.
\end{equation*}
See Section \ref{sec optimal} for more discussions.
\end{rmk}

To end this introduction, let us mention a few words on the proof of the main result (Theorem \ref{main thrm}). The key ingredient is a geometric inequality that we obtained in our previous work \cite{lx2016correspondences}, which is inspired by its correspondence in complex geometry. And the proof of this key ingredient involves some ideas from mass transport and real Monge-Amp\`{e}re equations.

\subsection{Organization}
In Section \ref{sec prel}, we present some basic definitions and set up some notations. In Section \ref{sec proof}, we discuss the proof of the main result and present some B\'ezout type inequalities in other settings.

\section{Preliminaries}\label{sec prel}

\subsection{Mixed volumes}
We start by giving some basic definitions and some notations. The general reference on convex geometry is Shneider's book \cite{schneiderconvex}.
Let $K_1, ...,K_r \subset \mathbb{R}^n$ be convex bodies and let $\vol(\cdot)$ be the volume function on $\mathbb{R}^n$, then there is a polynomial relation
\begin{equation*}
  \vol(t_1K_1 +...+t_r K_r) = \sum_{i_1 +...+i_r =n} \frac{n!}{i_1 !i_2 !...i_r !} V(K_1 ^{i_1},...,K_r ^{i_r}) t_1 ^{i_1}...t_r ^{i_r},
\end{equation*}
where $t_i \geq 0$. Then the coefficients $V(K_1 ^{i_1},...,K_r ^{i_r})$ define the mixed volumes. In particular, when $K_1 =...=K_r = K$ (up to translations), then the mixed volume is equal to $\vol(K)$.

\begin{lem}
With respect to the Hausdorff metric of compact sets, the mixed volume $V(K_1 ^{i_1},...,K_r ^{i_r})$ is continuous to the variables $K_1,...,K_r$.
\end{lem}

\begin{lem}
The mixed volume $V(K_1 ^{i_1},...,K_r ^{i_r})$ is symmetric to the variables, and is monotone with respect to the inclusion relation.
\end{lem}

\subsection{Mixed discriminants}\label{sec discr}
Similar to mixed volumes, the mixed discriminants can be defined as following. Let $M_1, ..., M_r$ be positive definite hermitian matrices, then for $t_i \geq 0$ we have
\begin{equation*}
  \det(t_1 M_1 +...+t_r M_r) = \sum_{i_1 +...+i_r =n} \frac{n!}{i_1 !i_2 !...i_r !} D(M_1 ^{i_1},...,M_r ^{i_r}) t_1 ^{i_1}...t_r ^{i_r}.
\end{equation*}
The coefficients $D(M_1 ^{i_1},...,M_r ^{i_r})$ define the mixed discriminants. %Equivalently, if $M^{(i)}$ denotes the $i$-th column of the matrix $M$, then
%\begin{equation*}
%  \frac{n!}{i_1 !i_2 !...i_r !} D(M_1 ^{i_1},...,M_r ^{i_r}) =\frac{1}{n!}\sum_{\sigma \in S(n)}\det(M_{\sigma(1)}^{(1)},...,M_{\sigma(n)}^{(n)}),
%\end{equation*}
%where $S(n)$ denotes the group of permutations of the numbers.

In our setting, from the viewpoints of complex geometry, it is useful to interpret the mixed discriminants as the wedge products of positive $(1,1)$ forms. Assume that $M=[a_{i\bar j}]$ be a positive definite hermitian matrix, then it determines a positive $(1,1)$ form on $\mathbb{C}^n$:
\begin{equation*}
  M\mapsto  \omega_M :=\sqrt{-1}\sum_{i, j} a_{i\bar j} dz^i \wedge d{\bar z} ^j.
\end{equation*}
Note that $\omega_M ^n = \det(M) \Phi $, where $\Phi$ is a volume form of $\mathbb{C}^n$. Denote the positive $(1,1)$ form associated to $M_i$ by $\omega_i$. Since
\begin{equation*}
  (t_1 \omega_1 +...+t_r \omega_r)^n = \sum_{i_1 +...+i_r =n}\frac{n!}{i_1 !i_2 !...i_r !} (\omega_1 ^{i_1}\wedge...\wedge \omega_r ^{i_r}) t_1 ^{i_1}...t_r ^{i_r},
\end{equation*}
we get $D(M_1 ^{i_1},...,M_r ^{i_r})=\omega_1 ^{i_1}\wedge...\wedge \omega_r ^{i_r} /\Phi$.

\section{B\'{e}zout type inequality}\label{sec proof}

In this section, we discuss the proof around Theorem \ref{main thrm}. First, we give an alternative method to prove Corollary \ref{main cor}. It actually follows from Diskant inequality.

\subsection{Proof of Corollary \ref{main cor}: using Diskant inequality}
Let $K, L$ be convex bodies in $\mathbb{R}^n$ with non-empty interior. Recall that the \emph{inradius} of $K$ relative to $L$ is defined by
\begin{equation*}
  r(K, L):=\max\{\lambda>0| \lambda L + t \subset K\ \ \textrm{for some}\ t\in \mathbb{R}^n\}.
\end{equation*}
Applying the Diskant inequality (see \cite[Section 7.2]{schneiderconvex}) to $K, L$ implies
\begin{align*}
  r(K, L)\geq \frac{V(K^{n-1}, L)^{1/n-1}-\left(V(K^{n-1}, L)^{n/n-1}-\vol(K)\vol(L)^{1/n-1}\right)^{1/n}}{\vol(L)^{1/n-1}}.
\end{align*}
If $K$ is not a scaling of $L$ for any translation, as in \cite[Section 5]{lx2016correspondences}, applying the generalized binomial formula to the second bracket in numerator yields
\begin{align*}
  &V(K^{n-1}, L)^{1/n-1}-\left(V(K^{n-1}, L)^{n/n-1}-\vol(K)\vol(L)^{1/n-1}\right)^{1/n}\\
  &= \sum_{k=1}^{\infty} \left( \begin{array}{c} 1/n \\ k \end{array}\right) (-1)^{k+1} \vol(K)^k\vol(L)^{k/n-1} V(K^{n-1}, L)^{\frac{1-kn}{n-1}}\\
  &= \frac{1}{n}\vol(K)\vol(L)^{1/n-1} V(K^{n-1}, L)^{-1} + \sum_{k=2}^{\infty}...\\
  & \geq \frac{1}{n}\vol(K)\vol(L)^{1/n-1} V(K^{n-1}, L)^{-1},
\end{align*}
because every term in the sum $\sum_{k=2}^{\infty} ...$ is non negative. Substituting, we obtain
\begin{equation}\label{ineq inradius}
  r(K, L)\geq \frac{\vol(K)}{nV(K^{n-1}, L)}.
\end{equation}
Note that the above inequality clearly holds when $K =cL$ (up to some translation). Thus in any case, we have $L \subseteq \frac{nV(L, K^{n-1})}{\vol(K)} K$ (up to some translation).

Now we are able to give the first proof of Corollary \ref{main cor}.

\begin{proof}
Using the continuity of mixed volumes, by taking limits we can assume all the convex bodies $K_1, ...,K_r, D$ have non-empty interior. Applying the above estimate of inradius to the pairs $(K_2, D),...,(K_r, D)$ yields, up to some translations,
\begin{equation*}
  K_i \subseteq \frac{nV(K_i, D^{n-1})}{\vol(D)} D \ \ \textrm{for every}\ \ 2\leq i \leq r.
\end{equation*}
By the monotone property of mixed volumes, the above inclusions immediately imply Corollary \ref{main cor}:
\begin{equation*}
   V(K_1, K_2,...,K_r, D^{n-r})\vol(D)^{r-1} \leq n^{r-1}\prod_{i=1} ^r V(K_i, D^{n-1}).
\end{equation*}
\end{proof}

\begin{rmk}
The inclusion $L \subseteq \frac{nV(L, K^{n-1})}{\vol(K)} K$ (up to some translation) corresponds exactly to the Morse type bigness criterion in algebraic geometry (see e.g. \cite[Chapter 8]{demaillyAGbook}). Actually, using Alexandrov body construction we have
$\vol(\{h_K -h_L\})\geq \vol(K) - nV(K^{n-1}, L)$, which can be considered as the ``algebraic Morse inequality'' in convex geometry (see \cite[Section 5.1]{lx2016correspondences} for more details).
\end{rmk}

\subsection{Proof of Theorem \ref{main thrm}: using reverse Khovanskii-Teissier inequality}
As a consequence of the above inradius estimate (\ref{ineq inradius}), it is easy to see that, for any convex bodies $K, L, M$ in $\mathbb{R}^n$ we have
\begin{equation}\label{ineq morse1}
  nV(K, L^{n-1})V(L, M^{n-1})\geq \vol(L)V(K, M^{n-1}).
\end{equation}
For instance, one can see \cite[Section 5]{lx2016correspondences}. Actually, as a consequence of \cite{weilinclusion} (see also \cite{lutwakinclusion}), the above inequality (\ref{ineq morse1}) implies (\ref{ineq inradius}). Thus they are equivalent.

\begin{rmk}
We call the above kind of inequality as ``reverse Khovanskii-Teissier inequality'' in \cite{lx2016correspondences}, because it gives us an upper bound of the mixed volume $V(K, M^{n-1})$. See also \cite{lehmxiao16convexity} for the discussion in an abstract setting from the viewpoints of convex analysis. Its K\"ahler geometry version first appeared in \cite{popovici2016sufficientbig}, which was proved using complex Monge-Amp\`{e}re equations.
\end{rmk}

Inspired by \cite{milman99masstransport} and the corresponding results in complex geometry, using methods from mass transport (see \cite[Theorem 5.9]{lx2016correspondences}), the inequality (\ref{ineq morse1}) can be generalized: for any convex bodies $K, L, M$ in $\mathbb{R}^n$ we have
\begin{equation*}
 \left(\begin{array}{c} n\\ k \end{array} \right) V(K^k, L^{n-k})V(L^k, M^{n-k})\geq \vol(L)V(K^k, M^{n-k}),
\end{equation*}
where $\left(\begin{array}{c} n\\ k \end{array} \right) = \frac{n!}{k!(n-k)!}$.

The above inequality can be generalized as following, which is enough to conclude Theorem \ref{main thrm}.

\begin{lem}\label{lem morse}
For any convex bodies $K, L, M_1, ..., M_{n-k}$ in $\mathbb{R}^n$, we have
\begin{equation}\label{ineq morsek}
 \left(\begin{array}{c} n\\ k \end{array} \right) V(K^k, L^{n-k})V(L^k, M_1, ..., M_{n-k})\geq \vol(L)V(K^k, M_1, ..., M_{n-k}).
\end{equation}
\end{lem}

\begin{proof}
By continuity of mixed volumes, we can assume all convex bodies have non-empty interior.

Before starting the proof, we point out that, for $k=1$ the inequality follows from (\ref{ineq morse1}) (thus is a consequence of Diskant inequality). This is a consequence of Minkowski's existence theorem (see e.g. \cite[Chapter 8]{schneiderconvex}): up to some translations, there exists a unique convex body $\widehat{M}$ such that $V(M_1, ..., M_{n-1}, \cdot)=V(\widehat{M}^{n-1},\cdot)$.

For general $k$, its proof is almost the same as 
\cite[Theorem 5.9]{lx2016correspondences}. So we just sketch the ingredients and the ideas. Inspired by \cite{milman99masstransport} where the authors reproved some of the Alexandrov-Fenchel inequalities by using mass transport, we apply a result of \cite{gromov1990convex}
and results from mass transport (see \cite{brenier1991polar, mccann95existence}). Then after solving a real Monge-Amp\`{e}re equation related to $L$, the desired geometric inequality of convex bodies can be reduced to an inequality for mixed discriminants -- more precisely, the mixed discriminants given by the Hessian of those convex functions defining the convex bodies. We omit the details for this reduction.

In our setting, the inequality needed for mixed discriminants can be stated as following: for any positive definite symmetric matrices $A, B, C_1, ..., C_{n-k}$ we have
\begin{equation*}\label{ineq disc}
  \left(\begin{array}{c} n\\ k \end{array} \right) D(A^k, B^{n-k})D(B^k, C_1, ..., C_{n-k})\geq \det(B)D(A^k, C_1, ..., C_{n-k}).
\end{equation*}
By the discussions in Section \ref{sec discr}, denote the associated positive $(1,1)$ forms to $A, B, C_1, ..., C_{n-k}$ by $\omega_A, \omega_B, \omega_1,...,\omega_{n-k}$, then the above inequality is equivalent to
\begin{equation*}
  \left(\begin{array}{c} n\\ k \end{array} \right) (\omega_A ^k \wedge \omega_B ^{n-k})(\omega_B^k \wedge \omega_1\wedge...\wedge\omega_{n-k})\geq \omega_B ^n (\omega_A^k \wedge \omega_1\wedge...\wedge\omega_{n-k}).
\end{equation*}
Now the proof is straightforward (see e.g. \cite[inequality (6)]{lx2016correspondences}). By changing the coordinates, we can assume that
\begin{equation*}
  \omega_B = \sqrt{-1}\sum_{j=1}^n dz^j \wedge d\bar{z}^j, \ \ \omega_A =\sqrt{-1}\sum_{j=1}^n \mu_j dz^j \wedge d\bar{z}^j,
\end{equation*}
and
\begin{equation*}
  \omega_1\wedge...\wedge\omega_{n-k} = (\sqrt{-1})^{n-k} \sum_{|I|=|J|=n-k} \Gamma_{IJ} dz^I \wedge d\bar{z}^J.
\end{equation*}
Denote by $\mu_J$ the product $\mu_{j_1}...\mu_{j_k}$ with index $J=(j_1<...<j_k)$ and denote by $J^c$ the complement index of $J$. Then it is easy to see
\begin{equation*}
  \left(\begin{array}{c} n\\ k \end{array} \right)\frac{\omega_A ^k \wedge \omega_B ^{n-k}}{\omega_B ^n} \cdot \frac{\omega_B^k \wedge \omega_1\wedge...\wedge\omega_{n-k}}{\omega_A^k \wedge \omega_1\wedge...\wedge\omega_{n-k}}= \frac{(\sum_{J}\mu_J)(\sum_K \Gamma_{KK})}{\sum_{J}\mu_J \Gamma_{J^c J^c}}\geq 1.
\end{equation*}

Thus we get the desired inequality for mixed discriminants, and this finishes the proof of the mixed volume inequality for convex bodies.

\end{proof}

\begin{rmk}
We remark that the above result (and its proof) is inspired by its correspondence in complex geometry. And the above inequality for positive $(1,1)$ forms is just a bit generalization of a pointwise estimate in \cite{popovici2016sufficientbig}, where the author used the inequality for $k=1$ and complex Monge-Amp\`{e}re equations to study Morse type inequality in the K\"ahler setting.
\end{rmk}

Now we can prove Theorem \ref{main thrm}.

\begin{proof}
Without loss of generalities, we assume all the convex bodies $K_1,...,K_r, D$ have non-empty interior. It is enough to show the case when $a_k$ is taken to be $a_1$. Applying Lemma \ref{lem morse} to $L=D, K= K_2$ implies
\begin{equation*}
  V(K_1 ^{a_1}, K_2 ^{a_2},...,K_r ^{a_r}, D^{n-|a|})\vol(D) \leq \left(\begin{array}{c} n\\ a_2 \end{array} \right)V(K_1 ^{a_1},D ^{a_2},...,K_r ^{a_r}, D^{n-|a|}) V(K_2 ^{a_2}, D^{n-a_2}).
\end{equation*}
Similarly, for the factor on the right hand side, we have
\begin{equation*}
  V(K_1 ^{a_1},D ^{a_2},...,K_r ^{a_r}, D^{n-|a|})\vol(D) \leq \left(\begin{array}{c} n\\ a_3 \end{array} \right)V(K_1 ^{a_1}, D ^{a_2}, D^{a_3},...,K_r ^{a_r}, D^{n-|a|}) V(K_3 ^{a_3}, D^{n-a_3}).
\end{equation*}
Repeating the procedure for $r-1$ times and multiplying these inequalities yield the desired B\'ezout type inequality:
\begin{equation*}
 \left(\begin{array}{c} n\\ a_1 \end{array} \right) V(K_1 ^{a_1}, K_2 ^{a_2},...,K_r ^{a_r}, D^{n-|a|})\vol(D)^{r-1} \leq \prod_{i=1} ^r \left(\begin{array}{c} n\\ a_i \end{array} \right) V(K_i ^{a_i}, D^{n-a_i}).
\end{equation*}
\end{proof}

\subsection{B\'{e}zout type inequality in other settings}

By the proof of the main result, it is clear that we have a similar inequality for mixed discriminants (or wedge products of positive forms).
\begin{prop}
Assume that $a_1,...,a_r \in \mathbb{N}$ and $|a|:=\sum_{i=1} ^r a_i \leq n$, then for any (semi-)positive definite symmetric matrices $M_1, ..., M_r, N$ and any integer $k$ satisfying $1\leq k\leq r$, we have
\begin{equation*}
 \left(\begin{array}{c} n\\ a_k \end{array} \right) D(M_1 ^{a_1}, M_2 ^{a_2},..., M_r ^{a_r}, N^{n-|a|})\det(N)^{r-1} \leq \prod_{i=1} ^r \left(\begin{array}{c} n\\ a_i \end{array} \right) D(M_i ^{a_i}, N^{n-a_i}).
\end{equation*}
\end{prop}

\begin{rmk}
It is clear that the B\'{e}zout type inequality for mixed discriminants is a ``determinant-trace'' type inequality. In the case when $M_1, ..., M_r$ are diagonal matrices, it is easy to see that
\begin{equation*}
   D(M_1 ^{a_1}, M_2 ^{a_2},..., M_r ^{a_r}, N^{n-|a|})\det(N)^{r-1} \leq \frac{(n!)^{r-1}(n-|a|)!}{\prod_{i=1} ^r (n-a_i)!}\prod_{i=1} ^r D(M_i ^{a_i}, N^{n-a_i}).
\end{equation*}
We are not sure if it also holds if $M_1, ..., M_r$ are not diagonal. For positive matrices $M_1, ..., M_r, N$ satisfying some other special conditions (e.g. ``doubly stochastic''), the constant can be estimated more effectively, see e.g. \cite{discrbound}, \cite{gurvitsDiscr}, \cite{gurvitsDisc1}.
\end{rmk}

From the correspondences between complex geometry and convex geometry (see \cite{lx2016correspondences}), we also have the following analogy on projective varieties.

\begin{prop}\label{thrm bezout divisor}
Let $X$ be a smooth projective variety of dimension $n$, defined over $\mathbb{C}$. Assume that $a_1,...,a_r \in \mathbb{N}$ and $|a|:=\sum_{i=1} ^r a_i \leq n$, then for any nef divisors $A_1, ..., A_r, D$ and any integer $k$ satisfying $1\leq k\leq r$, we have
\begin{equation*}
 \left(\begin{array}{c} n\\ a_k \end{array} \right) (A_1 ^{a_1}\cdot A_2 ^{a_2}\cdot...\cdot A_r ^{a_r}\cdot D^{n-|a|})(D^n)^{r-1} \leq \prod_{i=1} ^r \left(\begin{array}{c} n\\ a_i \end{array} \right) (A_i ^{a_i}\cdot D^{n-a_i}).
\end{equation*}
\end{prop}

\begin{rmk}
The above result is also true for compact K\"ahler manifolds when we replace the divisors by nef $(1,1)$ classes. It also holds true for pseudo-effective classes when we replace the usual intersections by movable intersections, after taking suitable Fujita approximations.
\end{rmk}

\begin{exmple}\label{rmk eq comp}
The above inequality is sharp in dimension 2: let $A_1, A_2$ be nef classes on a smooth projective surface such that $A_1 ^2 = A_2 ^2 =0$ and let $D=A_1 + A_2$, then $D^2 = 2 A_1 \cdot A_2$ and $A_1 \cdot D = A_2 \cdot D = A_1 \cdot A_2$, thus $(A_1\cdot A_2) D^2 = 2 (A_1 \cdot D) (A_2 \cdot D)$.
\end{exmple}

\subsection{Optimal constant} \label{sec optimal}
At the end of this section, let us mention a few words on the constant in the B\'ezout type inequality. For example, we consider the inequality for convex bodies. The best possible constant depends on the relations between $K_1, ..., K_r$ and $D$. For instance, if $K_1 = K_2 =...=K_r=K$ and $a_1 =...=a_r =1$, then we need to compare $V(K^{r}, D^{n-r}) \vol(D)^{r-1}$ and $V(K, D^{n-1})^r$. By log-concavity of the sequence $\{V(K^k , D^{n-k})\}_k$, we actually have
\begin{equation*}
  V(K^{r}, D^{n-r})\vol(D)^{r-1} \leq V(K,D^{n-1})^r.
\end{equation*}

Another example is the following.

\begin{thrm}
Assume that $a_1,...,a_r \in \mathbb{N}$ and that $M_1 ^{a_1}, ..., M_r ^{a_r}$ satisfy the relation $\mathcal{R}$: $|a|=\sum_{i=1} ^r a_i = n$.  Assume that $c_\mathcal{R}$ is a positive constant such that
\begin{equation*}
 D(M_1 ^{a_1}, M_2 ^{a_2},..., M_r ^{a_r})\det(N)^{r-1} \leq c_\mathcal{R} \prod_{i=1} ^r  D(M_i ^{a_i}, N^{n-a_i})
\end{equation*}
holds true for any choices of positive matrices $M_1,...,M_r, N$. Then
\begin{itemize}
  \item for any convex bodies $K_1, ..., K_r, D$ we have
  \begin{equation*}
    V(K_1 ^{a_1}, K_2 ^{a_2},..., K_r ^{a_r})\vol(D)^{r-1} \leq c_\mathcal{R} \prod_{i=1} ^r  V(K_i ^{a_i}, D^{n-a_i});
  \end{equation*}
  \item for any nef classes $A_1, ..., A_r, D$ we have
  \begin{equation*}
    (A_1 ^{a_1} \cdot A_2 ^{a_2}\cdot...\cdot A_r ^{a_r})\vol(D)^{r-1} \leq c_\mathcal{R} \prod_{i=1} ^r  (A_i ^{a_i} \cdot D^{n-a_i}).
  \end{equation*}
\end{itemize}
\end{thrm}

The method of the proof is the same as Lemma \ref{lem morse} -- reduce the global inequality to a local one. (Within the same spirit, it is inspiring to compare with the proofs in \cite{milman99masstransport}, \cite{FX14}, \cite{popovici2016sufficientbig} and \cite{lx2016correspondences}.)

\begin{proof}
For simplicity, let us focus on the complex geometry situation. Without loss of generalities, we assume that all classes are K\"ahler. We use the same symbols $A_1,...,A_r, D$ to denote K\"ahler metrics in the corresponding classes. By \cite{Yau78}, we can solve the following equation:
\begin{equation}\label{eq MA}
  D_u ^n = c A_1 ^{a_1} \wedge A_2 ^{a_2}\wedge...\wedge A_r ^{a_r},
\end{equation}
where $D_u$ is a K\"ahler metric in the class $D$, $c=D^n / A_1 ^{a_1} \cdot A_2 ^{a_2}\cdot...\cdot A_r ^{a_r}$ is a constant. Then we have
\begin{align*}
  \prod_{i=1} ^r  (A_i ^{a_i} \cdot D^{n-a_i}) &=\prod_{i=1} ^r \int  A_i ^{a_i} \wedge D_u ^{n-a_i}\\
  &\geq \left(\int \left(\prod_{i=1} ^r \frac{A_i ^{a_i} \wedge D_u ^{n-a_i}}{D_u ^n}\right)^{1/r} D_u ^n \right)^r\\
  &\geq \left(\int \left(\frac{(A_1 ^{a_1}\wedge ...\wedge A_r ^{a_r})(D_u ^n)^{r-1}}{c_{\mathcal{R}}(D_u ^n)^r}\right)^{1/r} D_u ^n \right)^r\\
  &=(A_1 ^{a_1}\cdot ...\cdot A_r ^{a_r})(D^n)^{r-1}/c_{\mathcal{R}},
\end{align*}
where the first inequality follows from H\"{o}lder inequality, the second inequality follows from the property of $c_{\mathcal{R}}$, and the last inequality follows from (\ref{eq MA}). The above estimates clearly imply
 \begin{equation*}
    (A_1 ^{a_1} \cdot A_2 ^{a_2}\cdot...\cdot A_r ^{a_r})\vol(D)^{r-1} \leq c_\mathcal{R} \prod_{i=1} ^r  (A_i ^{a_i} \cdot D^{n-a_i}).
  \end{equation*}

For convex bodies, we use the same argument as in Lemma \ref{lem morse} (see \cite[Section 5]{lx2016correspondences} for more details). Then we need to estimate some integrals over $\mathbb{R}^n$, and similarly this can be reduced to the estimate for mixed discriminants.
\end{proof}

\begin{rmk}
The best possible constant $c_\mathcal{R}$ may vary when $K_1,...,K_r, D$ (or $A_1, ...,A_r, D$) satisfy additional relations, e.g. some $K_i$ (or $A_i$) are proportional to $D$, as indicated by the discussion at the beginning. See \cite{gurvits} for some other type of upper bound for mixed discriminants. 
As for the possible improvement in the general setting of Theorem \ref{main thrm}, in the reduction from the global one to local one, it should be related to Hessian type equations.
\end{rmk}

\bibliography{reference}
\bibliographystyle{alpha}

\bigskip

\bigskip

\noindent
\textsc{Department of Mathematics, Northwestern University,
Evanston, IL 60208, USA}\\
\noindent
\verb"Email: jianxiao@math.northwestern.edu"
\end{document}